\documentclass[12pt]{amsart}
\usepackage[all]{xy}
\usepackage{pdfsync}
 \usepackage[colorlinks=false]{hyperref}
\usepackage{amssymb} 
\usepackage{amsmath} 
\usepackage{graphicx} 
\usepackage[english]{babel} 
\usepackage{epsfig}
\usepackage{multirow}
\usepackage{color}
\usepackage{enumerate}
\usepackage[T1]{fontenc}
\usepackage{aurical, pbsi,calligra}
\swapnumbers

\theoremstyle{definition}
\newtheorem{definition}{Definition}[section]

\theoremstyle{plain}
\newtheorem{lemma}[definition]{Lemma}
\newtheorem{prp}[definition]{Proposition}
\newtheorem{corollary}[definition]{Corollary}
\newtheorem{theorem}[definition]{Theorem}
\usepackage[all]{xy}
\usepackage{pdfsync}
\usepackage{ytableau}
\newcommand\myeq{\mathrel{\overset{\makebox[0pt]{\mbox{\normalfont\tiny\sffamily def}}}{=}}}

\newcommand{\sh} {\llcorner\!\shortmid \!\lrcorner}
\newcommand{\sudda}[1]{}
\begin{document}

\title{On the speciality of Tortkara algebras}

\author{A.S. Dzhumadil'daev}

\address{Institute of Mathematics and Mathematical Modeling, Almaty, Kazakhstan}

\email{dzhuma@hotmail.com}

\author{N.A. Ismailov}

\address{Instituto de Matem\'atica e Estat\'istica, Universidade de S\~ao Paulo,
S\~ao Paulo, SP, Brazil}

\email{nurlan.ismail@gmail.com}

\author{F. A. Mashurov}

\address{Institute of Mathematics and Mathematical Modeling, Almaty, Kazakhstan}

\email{f.mashurov@gmail.com}


\thanks{The authors were supported by grant AP05131123 "Cohomological and structural problems of non-associative algebras" and the second author was supported by the FAPESP Proc.2017/21429-6}

 \maketitle

\begin{abstract} A criterion for elements of free Zinbiel algebras to be Lie or Jordan is established. This criterion is used in studying speciality problems of Tortkara algebras. We construct a base of free special Tortkara algebras. Furthermore, we prove analogues of classical Cohn's and Shirshov's theorems for Tortkara algebras.
\end{abstract}

\section{\label{nn}\ Introduction}

An algebra with identity 
\begin{equation}\label{Zinbiel identity}
a\circ(b\circ c)=(a\circ b+b\circ a)\circ c\end{equation}
is called {\it(right)-Zinbiel} algebra. These algebras were introduced by J-L.Loday in \cite{Loday} as a dual of Leibniz algebras. In some papers Zinbiel algebras are called {\it dual Leibniz}, {\it chronological} or {\it pre-commutative algebras} \cite {Agr-Gam1},\cite {Agr-Gam2},\cite {Kawski},\cite{Kolesnikov}.  

Anticommuative algebra with so-called ${\it Tortkara}$ identity  
\begin{equation}\label{Tortkara identity}(ab)(cb)=J(a,b,c)b
\end{equation}
is called {\it Tortkara} algebra.
Over a field of characteristic different from  two (\ref{Tortkara identity}) has the following multi-linear form
\begin{equation}\label{linear-Tortkara identity}(ab)(cd)+(ad)(cb)=J(a,b,c)d+J(a,d,c)b.
\end{equation}
Tortkara algebra was  defined in \cite{Dzhumadil'daev}, as minus-algebra $A^{(-)}=(A,[,],+)$ of Zinbiel algebra $(A,\circ,+)$, in other words,
any Zinbiel algebra  under the commutator product $[x,y]=x\circ y-y\circ x$ satisfies the Tortkara identity. 
Any Zinbiel algebra under the anti-commutator product $\{x,y\}=x\circ y+y\circ x$ is commutative and associative  \cite{Loday}.

Metabelian Lie and dual mock-Lie algebras are  examples of Tortkara algebras, see \cite{Zusmanovich}. Let $A$ be an algebra with identities
\begin{equation}\label{Tortkara example}
a(bc)=b(ac), \quad ([a,c],b,a)+([b,a],c,a)+([c,b],a,a)=0.
\end{equation}
It is not a Zinbiel algebra, but
its minus-algebra $A^{(-)}$ is a Tortkara algebra \cite{Dzhumadil'daev}.

As far as connections of Tortkara algebras with other classes of algebras we have the following facts. 
The variety of Tortkara algebras $\mathcal{T}$ is not subvariety of the variety of binary-Lie algebras. Consequently,
it is not a subvariety of variety of Malcev algebras. 
In \cite{Sokolov-Svinolupov} there was studied a variety of algebras defined by the following identity 
$$[a,b,cd]-[a,b,c]d-c[a,b,d]=0.$$
It contains Jordan, Lie, right-symmetric and LT-algebras. Tortkara algebras are out of this variety.

Tortkara algebra $B$ is called {\it special} if there exists a Zinbiel algebra $A$ such that $B$ is a subalgebra of $A^{(-)},$ otherwise it is called {\it exceptional}.

Let $A={\bf C}[x]$ be an algebra with multiplication
$$a\star b=b\int_0^x(\int^x_0a\,dx)dx.$$ Then $(A,\star)$ is not a Zinbiel algebra,
but $A^{(-)}$ is a Tortkara algebra \cite{Dzhumadil'daev}. It is an example of special Tortkara algebra 
 (see Remark 1 at the end of our paper). 

In our paper we exhibit a base of free special Tortkara algebras in terms of polynomials of free Zinbiel algebras. 

Let $\mathcal{ST}$  be the class of special Tortkara algebras. By $\overline{\mathcal{ST}}$ denote the homomorphic closure of $\mathcal{ST}$. In general, $$\mathcal{ST}\subseteq \overline{\mathcal{ST}}\subseteq\mathcal{T}$$ and
$\overline{\mathcal{ST}}$ is the smallest variety that contains $\mathcal{ST}.$

In \cite{Kolesnikov} P. Kolesnikov has proved that 
$$\mathcal{ST}\subsetneqq\overline{\mathcal{ST}}.$$
He has constructed a Tortkara algebra on four generators which is in $\overline{\mathcal{ST}}$ but not in $\mathcal{ST}.$  Furthermore, he has asked the question about the maximal number of free generators for which all homomorphic images of free special Tortkara algebra are special. In this paper we show that there exists an algebra in $\overline{\mathcal{ST}}$ on three generators which is not in $\mathcal{ST},$ and any algebra in $\overline{\mathcal{ST}}$ on two generators is also in $\mathcal{ST}.$ This is the analogue of classical Cohn's theorem \cite{Cohn} in Jordan theory for free special Tortkara algebras. The analogue of this theorem for a free special Jordan dialgebra is studied in \cite{Voronin}.

We further prove that a free Tortkara algebra on two generators is special. As a consequence, taking account the above result, we have that any Tortkara algebra on two generators is special. This is the analogue of Shirshov's theorem in theory of Jordan algebras. Shirshov's theorem for Malcev algebras was considered in \cite{Kornev}.  

An idenitity is called {\it special} if it holds in any Tortkara algebra in $\overline{\mathcal{ST}}$ but does not hold in some  Tortkara algebras. In general, we still do not know whether there exists a special identity. In \cite{Bremner},  M. Bremner by computer algebraic methods has studied special identities in terms of Tortkara triple product $[a,b,c]=[[a,b],c]$ in a free Zinbiel algebra and discovered one identity in degree 5 and one identity in degree 7 which do not follow from the triple identities of lower degree. We prove  that there is no special identity in two variables. It is known that there is no special identity up to degree seven in case a field has zero characteristic. 

All our algebras  are considered over a field ${\bf K}$ of characteristic 0.

\section{\label{mainresults} Main results}

\subsection{Definitions and notations}
Let $Zin(X)$ be a free Zinbiel algebra on a set $X.$ For $a_1,\ldots,a_n\in Zin(X)$ denote by $a_1a_2\cdots a_n$ a left-bracketed element $(\cdots(a_1\circ a_2)\cdots)\circ a_n.$ 
In \cite{Loday} there was proved that the following set of elements 
$$\mathcal{V}(X)=\cup_n\{x_{i_1}x_{i_2}\cdots x_{i_n}|x_1,\ldots x_{i_n}\in X\}$$
forms a base of the free Zinbiel algebra $Zin(X).$

Define a linear map $p:Zin(X)\rightarrow Zin(X)$ on base elements
as follows
$$p(x_i)=-x_i,$$ 
$$p(x_ix_j)=x_jx_i,$$
$$p( x_{i_1}x_{i_2}\cdots x_{i_{m}}yz)=x_{i_1}x_{i_2}\cdots x_{i_m}zy, \quad m\geq1.$$
 where $y,z\in X.$

For  $a\in Zin(X)\setminus X$ set
\begin{align*}\overline{a}
&\myeq \ 
a-p(a)\end{align*}  

Since $p^2=id,$ it is clear that
$$p(\overline{a})=-\overline{a}.$$

Let $n>1$ be an integer and let $\Gamma$ be set of sequences $\alpha=i_1 \cdots i_{n-1} i_n$ such that $i_{n-1}< i_n.$ 
For $\alpha=i_1\ldots i_{n-1}i_n\in \Gamma$ set 
$$x_\alpha={x_{i_1}\cdots x_{i_{n-1}}x_{i_n}}.$$

We call elements of the form  $\overline{x_\alpha}$, where $\alpha\in \Gamma,$ {\it skew-right-commutative} or shortly {\it skew-rcom} elements of $Zin(X).$

For example, if $X=\{a,b,u,v\},$ then
\begin{center}

$\overline{\strut(ab)(uv)}=(ab)(uv)-p((ab)(uv))=((ab)u)v+(u(ab))v-p(((ab)u)v+(u(ab))v)=$\\

$abuv+uabv+aubv-p(abuv+uabv+aubv)=abuv+uabv+aubv-abvu-uavb-auvb=$\\

$\overline{\strut abuv}+\overline{uabv}+\overline{aubv},$

$\overline{\strut x_{3124}}=uabv-uavb.$
\end{center}

An element $u\in Zin(X)$ is called {\it Lie element} if it can be presented as a word on $X$ under Lie product $[a,b]=ab-ba.$ Similarly, 
an element $u\in Zin(X)$ is called {\it Jordan element} if it can be presented as a word on $X$ in terms of Jordan product $\{a,b\}=ab+ba.$ 

Let $ST(X)$ be the free special Tortkara algebra on $X$ under Lie commutator, i.e., subalgebra of $Zin(X)^{(-)}=(Zin(X), [\;,\;])$ generated by $X.$ Let $J(X)$ be  subalgebra of $Zin(X)^{(+)}=(Zin(X), \{\;,\;\})$ generated by $X.$

Define {\it Dynkin map} $D:Zin(X)\rightarrow Zin(X)$ on base elements
as follows
$$D: x_{i_1}x_{i_2}\cdots x_{i_n}\mapsto\{\{\cdots\{x_{i_1},x_{i_2}\}\cdots \},x_{i_n}\}.$$

\subsection{Formulations of main results}

\begin{theorem}\label{Lie Criterion}
Let $f$ be a Zinbiel element of $Zin(X).$ Then $f$ is a Lie element if and only if $p(f)=-f.$  
\end{theorem}
\begin{theorem}\label{Skew element basis}
The set of skew-rcom elements $\overline{x_{\alpha},}$ where $ \alpha\in \Gamma,$ forms  base of $ST(X).$

Let   $ST(X)_{m_1,\ldots,m_q}$ be the homogenous part of $ST(X)$ generated by $m_i$ generators $x_i$ where $i=1,\ldots,q.$ 
Then  
$$dim\,ST(X)_{m_1,\ldots,m_q}=\sum_{i<j}\frac{(n-2)!}{m_1!\cdots m_q!}m_im_j$$
where $n=m_1+\dots+m_q.$ In particlular, the multilinear part of $ST(X)$ has dimension $\frac{q!}{2}.$ 
\end{theorem}

\begin{corollary}\label{On last two elements in skew rcom}
Let $a,b,c\in Zin(X).$ If $b$ and $c$ are Lie, then $abc-acb$ and $bc-cb$  are Lie.
\end{corollary}

\begin{theorem}\label{a Jordan criterion}
Let $f$ be a homogenous Zinbiel  element of degree $n$ in $Zin(X).$ Then $f$ is a Jordan element  if and only if $D(f)=n!f.$  
The algebra  $J(X)$ is isomorphic to polynomial algebra $K[X].$ 
\end{theorem}

Let $T(X)$ be a free Tortkara algebra generated by a set $X.$
\begin{theorem}\label{Speciality of $T({x,y})$}
The free Tortkara algebra $T(\{x,y\})$ is special.
\end{theorem}

The next theorem is an analogue of Cohn's theorem on speciality of homomorphic images of special Jordan algebras in two generators \cite{Cohn}. 

\begin{theorem}\label{speciality criterion on two generators}
Any homomorphic image of a free special Tortkara algebra on two generators is special.
For three generators case this statement is not true: a homomorphic image of special Tortkara algebra with three generators might be non special.
\end{theorem}

\begin{corollary}\label{Shirshov's theorem}
Any Tortkara algebra on two generators is special.
\end{corollary}
\begin{proof} It follows from Theorems \ref{Speciality of $T({x,y})$} and \ref{speciality criterion on two generators}.
\end{proof}
This result is an analogue of Shirshov theorem for Jordan algebras \cite{ZSSS}.

\section{\label{nn}\ Skew-right-commutative Zinbiel elements  and Lie elements in $Zin(X)$} 

\subsection{Shuffle permutations}

Let $Sh_{m,n}$ be set of shuffle permutations, i.e.,
$$Sh_{m,n}=\{\sigma\in S_{m+n} | \sigma(1)<\cdots<\sigma(m), \sigma(m+1)<\cdots\sigma(m+n)\}.$$
For any positive integers  $i_1,\ldots,i_m,j_1,\ldots,j_n$ denote by 
$Sh(i_1\ldots i_m;j_1\ldots j_n)$ set of  sequences $\alpha=\alpha_1\ldots\alpha_{n+m}$ constructed by 
shuffle permutations $\sigma\in Sh_{m,n}$  by changing $\alpha_{\sigma(l)}$ to $i_l$ if $l\le m$ and to $j_{l-m}$ if $m<l\le m+n.$

For example,
$$Sh(12;34)=\{1234, 1324, 3124, 1342, 3142,3412\},$$
$$Sh(23;41)=\{2341, 2431, 4231, 2413, 4213,4123\}.$$
The following proposition was proved in \cite{Loday} for free left-Zinbiel algebras, below we give it for free right-Zinbiel algebras.
\begin{prp}\label{Loday's prop} 
$$(x_{i_1}\cdots x_{i_p})\circ (x_{j_1}\cdots x_{j_q})=\sum_{\sigma\in Sh(i_1\ldots i_p; j_1\ldots j_{q-1})}x_{\sigma(1)}\cdots x_{\sigma(p+q-1)}x_{j_q}.$$
\end{prp} 

For any  two base elements $u=x_{i_1}\cdots x_{i_p}, v=x_{i_{p+1}}\cdots x_{i_{p+q}}\in Zin(X)$ define their {\it shuffle product} by 
$$u\sh v=\sum_{\sigma\in Sh(i_1\ldots i_p;i_{p+1}\ldots i_{p+q})}x_{\sigma(1)}\cdots x_{\sigma(p+q)}.$$

\begin{prp}\label{shuffle product} The shuffle product on $Zin(X)$ has the following properties

\begin{itemize}
\item[ \bf a.] the shuffle product is commutative and associative
$$a\sh b=b\sh a, \quad (a\sh b)\sh c=a\sh (b\sh c), \quad \mbox{for any} \quad a,b,c \in Zin(X).$$
\item[\bf b.]
$(x_{i_1}\cdots x_{i_p})\circ (x_{j_1}\cdots x_{j_q})=\begin{cases}  	
x_{i_1}\cdots x_{i_p}x_{j_1} \quad \mbox{for} \quad q=1, \\ (x_{i_1}\cdots x_{i_p}\sh x_{j_1}\cdots x_{q_{j-1}})\circ x_{j_q}  \quad \mbox{for} \quad q>1. \end{cases}$
 
\item[\bf c.] 
$(x_{i_1}\cdots x_{i_p})\sh (x_{j_1}\cdots x_{j_q})=$ 

$\begin{cases}  	
x_{i_1}x_{j_1}+x_{j_1}x_{i_1} \quad \mbox{for} \quad p=q=1, \\ (x_{i_1}\sh x_{j_1}\cdots x_{j_{q-1}})\circ x_{j_q}+x_{j_1}\cdots x_{j_q}x_{i_1}  \quad \mbox{for} \quad p=1,q>1,\\(x_{i_1}\cdots x_{i_{p-1}} \sh x_{j_1}\cdots x_{j_q})\circ x_{i_p}+(x_{i_1}\cdots x_{i_p}\sh x_{j_1}\cdots x_{j_{q-1}})\circ x_{j_q}  \quad \mbox{for} \quad p,q>1. \end{cases}$
\end{itemize}
\end{prp}

For example, 
\begin{center}
$(ab)\circ (cd)=(abc+acb+cab)  d=(ab\sh c)\circ d,$
$(ab)\sh (cd)=abcd+acbd+cabd+acdb+cadb+cdab=$
$(abc+acb+cab)d+(acd+cad+cda)b=(ab\sh c)d+(a\sh cd)b.$
\end{center}
\begin{proof}
All these properties follow from Proposition \ref{Loday's prop} and the definition of the shuffle product.
\end{proof}

\subsection{Products of skew-right-commutative elements} 

\begin{lemma}\label{Skew to Skew} Zinbiel product of skew-right-commutative elements can be presented as follows
$$\overline{\strut x_{i_1}\cdots x_{i_m}}\circ\overline{\strut x_{j_1}x_{j_2}}=\overline{\strut x_{i_1}\cdots x_{i_m}x_{j_1}x_{j_2}}-\overline{\strut x_{i_1}\cdots x_{i_{m-2}}x_{i_m}x_{i_{m-1}}x_{j_1}x_{j_2}}+$$$$(x_{i_1}\cdots x_{i_{m-1}}\sh x_{j_1})x_{i_m}x_{j_2}-(x_{i_1}\cdots x_{i_{m-2}}x_{i_{m}}\sh x_{j_1})x_{i_{m-1}}x_{j_2}-$$$$(x_{i_1}\cdots x_{i_{m-1}}\sh x_{j_2})x_{i_m}x_{j_1}+(x_{i_1}\cdots x_{i_{m-2}}x_{i_m}\sh x_{j_2})x_{i_{m-1}}x_{j_1}$$
and $$\overline{\strut x_{i_1}\cdots x_{i_m}}\circ\overline{\strut x_{j_1}\cdots x_{j_n}}=$$
$$\overline{\strut(x_{i_1}\cdots x_{i_m}\sh x_{j_1}\cdots x_{j_{n-2}})x_{j_{n-1}}x_{j_n}}-\overline{\strut(x_{i_1}\cdots x_{i_{m-2}} x_{i_m}x_{i_{m-1}}\sh x_{j_1}\cdots x_{j_{n-2}})x_{j_{n-1}}x_{j_n}}+$$
$$(x_{i_1}\cdots x_{i_{m-1}}\sh x_{j_1}\cdots x_{j_{n-1}})x_{i_m}x_{j_n}-(x_{i_1}\cdots x_{i_{m-2}} x_{i_m}\sh x_{j_{1}}\cdots x_{j_{n-1}})x_{i_{m-1}}x_{j_n}-$$
$$(x_{i_1}\cdots x_{i_{m-1}}\sh x_{j_1}\cdots x_{j_{n-2}}x_{j_n})x_{i_m}x_{j_{n-1}}+(x_{i_1}\cdots x_{i_{m-2}}x_{i_m} \sh x_{j_1}\cdots x_{j_{n-2}}x_{j_n})x_{i_{m-1}}x_{j_{n-1}},$$ where $m\geq2,n\geq3.$
\end{lemma}

\begin{proof}
$$\overline{\strut x_{i_1}\cdots x_{i_m}}\circ\overline{\strut x_{j_1}x_{j_2}}=x_{i_1}\cdots x_{i_m}\circ x_{j_1}x_{j_2}-x_{i_1}\cdots x_{i_{m-2}}x_{i_m}x_{i_{m-1}}\circ x_{j_1}x_{j_2}-$$$$x_{i_1}\cdots x_{i_m}\circ x_{j_2}x_{j_1}+x_{i_1}\cdots x_{i_{m-2}}x_{i_m}x_{i_{m-1}}\circ x_{j_2}x_{j_1}=$$
(by part {\bf b} of Proposition \ref{shuffle product})	
$$(x_{i_1}\cdots x_{i_m}\sh x_{j_1})x_{j_2}-(x_{i_1}\cdots x_{i_{m-2}}x_{i_m}x_{i_{m-1}}\sh x_{j_1})x_{j_2}-$$
$$(x_{i_1}\cdots x_{i_m}\sh x_{j_2})x_{j_1}+(x_{i_1}\cdots x_{i_{m-2}}x_{i_m}x_{i_{m-1}}\sh x_{j_2})x_{j_1}=$$
(by the definitions of shuffle product and skew-rcom elements)
$$\overline{\strut x_{i_1}\cdots x_{i_m}x_{j_1}x_{j_2}}-\overline{\strut x_{i_1}\cdots x_{i_{m-2}}x_{i_m}x_{i_{m-1}}x_{j_1}x_{j_2}}+$$$$(x_{i_1}\cdots x_{i_{m-1}}\sh x_{j_1})x_{i_m}x_{j_2}-(x_{i_1}\cdots x_{i_{m-2}}x_{i_m}\sh x_{j_1})x_{i_{m-1}}x_{j_2}-$$$$(x_{i_1}\cdots x_{i_{m-1}}\sh x_{j_2})x_{i_m}x_{j_1}+(x_{i_1}\cdots x_{i_{m-2}}x_{i_m}\sh x_{j_2})x_{i_{m-1}}x_{j_1}.$$
Let $n\geq3.$
$$\overline{\strut x_{i_1}\cdots x_{i_m}}\circ\overline{\strut x_{j_1}\cdots x_{j_n}}=$$
$$ x_{i_1}\cdots x_{i_m}\circ x_{j_1}\cdots x_{j_n} -x_{i_1}\cdots x_{i_{m-2}}x_{i_m} x_{i_{m-1}}\circ x_{j_1}\cdots x_{j_n}-$$
$$x_{i_1}\cdots x_{i_m}\circ x_{j_1}\cdots x_{j_{n-2}}x_{j_n}x_{j_{n-1}}+x_{i_1}\cdots x_{i_{m-2}} x_{i_m} x_{i_{m-1}}\circ x_{j_1}\cdots x_{j_{n-2}} x_{j_n}x_{j_{n-1}} =$$
(by part {\bf b} of Proposition \ref{shuffle product})	
$$(x_{i_1}\cdots x_{i_m}\sh x_{j_1}\cdots x_{j_{n-1}})x_{j_n}-(x_{i_1}\cdots x_{i_{m-2}}x_{i_m} x_{i_{m-1}}\sh x_{j_1}\cdots x_{j_{n-1}})x_{j_n}-$$
$$(x_{i_1}\cdots x_{i_m}\sh x_{j_1}\cdots x_{j_{n-2}}x_{j_n})x_{j_{n-1}}+(x_{i_1}\cdots x_{i_{m-2}}x_{i_m} x_{i_{m-1}}\sh x_{j_1}\cdots x_{j_{n-2}}x_{j_n})x_{j_{n-1}} =$$
(by part {\bf c} of Proposition \ref{shuffle product})	
$$(x_{i_1}\cdots x_{i_{m-1}}\sh x_{j_1}\cdots x_{j_{n-1}})x_{i_m}x_{j_n}+(x_{i_1}\cdots x_{i_m}\sh x_{j_1}\cdots x_{j_{n-2}})x_{j_{n-1}}x_{j_n}-$$
$$(x_{i_1}\cdots x_{i_{m-2}} x_{i_m}\sh x_{j_1}\cdots x_{j_{n-1}})x_{i_{m-1}}x_{j_n}-(x_{i_1}\cdots x_{i_{m-2}} x_{i_m}x_{i_{m-1}}\sh x_{j_1}\cdots x_{j_{n-2}})x_{j_{n-1}}x_{j_n}-$$
$$(x_{i_1}\cdots x_{i_{m-1}}\sh x_{j_1}\cdots x_{j_{n-2}}x_{j_n})x_{i_m}x_{j_{n-1}}-(x_{i_1}\cdots x_{i_m}\sh x_{j_1}\cdots x_{j_{n-2}})x_{j_n}x_{j_{n-1}}+$$
$$(x_{i_1}\cdots x_{i_{m-2}}x_{i_m} \sh x_{j_1}\cdots x_{j_{n-2}}x_{j_n})x_{i_{m-1}}x_{j_{n-1}}+(x_{i_1}\cdots x_{i_{m-2}}x_{i_m} x_{i_{m-1}}\sh x_{j_1}\cdots x_{j_{n-2}})x_{j_n}x_{j_{n-1}}=$$
$$\overline{\strut(x_{i_1}\cdots x_{i_m}\sh x_{j_1}\cdots x_{j_{n-2}})x_{j_{n-1}}x_{j_n}}-\overline{\strut(x_{i_1}\cdots x_{i_{m-2}} x_{i_m}x_{i_{m-1}}\sh x_{j_1}\cdots x_{j_{n-2}})x_{j_{n-1}}x_{j_n}}+$$
$$(x_{i_1}\cdots x_{i_{m-1}}\sh x_{j_1}\cdots x_{j_{n-1}})x_{i_m}x_{j_n}-(x_{i_1}\cdots x_{i_{m-2}} x_{i_m}\sh x_{j_1}\cdots x_{j_{n-1}})x_{i_{m-1}}x_{j_n}-$$
$$(x_{i_1}\cdots x_{i_{m-1}}\sh x_{j_1}\cdots x_{j_{n-2}}x_{j_n})x_{i_m}x_{j_{n-1}}+(x_{i_1}\cdots x_{i_{m-2}}x_{i_m} \sh x_{j_1}\cdots x_{j_{n-2}}x_{j_n})x_{i_{m-1}}x_{j_{n-1}}.$$
\end{proof}
	
\begin{lemma}\label{SkewRcomprudctgenerator}
$$[\overline{\strut x_{i_{1}}\cdots x_{i_{m-1}}x_{i_{m}}},x_{j_{1}}]=$$
\begin{center}
$\begin{cases}  	
\overline{\strut x_{i_1}x_{i_2}x_{j_1}}-\overline{\strut x_{i_2}x_{i_1}x_{j_1}}-\overline{\strut x_{j_1}x_{i_1}x_{i_2}}\quad \mbox{for} \quad m=2, \\ \overline{\strut x_{i_{1}}\cdots x_{i_{m}} x_{j_1}}-\overline{\strut x_{i_{1}}\cdots x_{i_{m-2}}x_{i_{m}}x_{i_{m-1}} x_{j_1}}-\overline{\strut (x_{j_1}\sh x_{i_{1}}\cdots x_{i_{m-2}})x_{i_{m-1}}x_{i_m}} \quad \mbox{for} \quad m>2.\end{cases}$
\end{center}
\end{lemma}	
\begin{proof} 
The first part of the statement can be obtained by direct calculations. Suppose $m>2.$ Then
$$[\overline{\strut x_{i_{1}}\cdots x_{i_{m}}},x_{j_1}]=\overline{\strut x_{i_{1}}\cdots x_{i_{m}}}\circ x_{j_1}-x_{j_1}\circ\overline{\strut x_{i_{1}}\cdots x_{i_m}}=$$
$$(x_{i_{1}}\cdots x_{i_m})\circ x_{j_1}-(x_{i_{1}}\cdots x_{i_{m-2}}x_{i_m}x_{i_{m-1}})\circ x_{j_1}-$$$$x_{j_1}\circ(x_{i_{1}}\cdots x_{i_m})+x_{j_1}\circ(x_{i_{1}}\cdots x_{i_{m-2}}x_{i_m}x_{i_{m-1}})=$$
(by part {\bf b} of Proposition \ref{shuffle product})	
$$x_{i_{1}}\cdots x_{i_m}x_{j_1}-x_{i_{1}}\cdots x_{i_{m-2}}x_{i_m}x_{i_{m-1}} x_{j_1}-$$
$$(x_{j_1}\sh x_{i_{1}}\cdots x_{i_{m-1}})x_{i_m}+(x_{j_1}\sh x_{i_{1}}\cdots x_{i_{m-2}}x_{i_m})x_{i_{m-1}}=$$
(by part {\bf c} of Proposition \ref{shuffle product})	
$$x_{i_{1}}\cdots x_{i_m} x_{j_1}-x_{i_{1}}\cdots x_{i_{m-2}}x_{i_m}x_{i_{m-1}} x_{j_1}-$$$$x_{i_{1}}\cdots x_{i_{m-1}}x_{j_1}x_{i_m}-(x_{j_1}\sh x_{i_{1}}\cdots x_{i_{m-2}})x_{i_{m-1}}x_{i_m}+$$$$x_{i_{1}}\cdots x_{i_{m-2}}x_{i_m}x_{j_1}x_{i_{m-1}}+(x_{j_1}\sh x_{i_{1}}\cdots x_{i_{m-2}})x_{i_m}x_{i_{m-1}}=$$
$$\overline{\strut x_{i_1}\cdots x_{i_m} x_{j_1}}-\overline{\strut x_{i_1}\cdots x_{i_{m-2}}x_{i_m}x_{i_{m-1}} x_{j_1}}-\overline{\strut(x_{j_1}\sh x_{i_{1}}\cdots x_{i_{m-2}})x_{i_{m-1}}x_{i_m}}.$$
\end{proof}

\begin{lemma}\label{commutator of Skews}	
Commutator of skew-right-commutative elements is a linear combination of skew-right-commutative elements.
\end{lemma}
\begin{proof} If skew-right-commutative elements have degree 2, then  
a straightforward calculation shows that $$[\overline{\strut x_{i_1}x_{i_2}},\overline{\strut x_{j_1}x_{j_2}}]=\overline{\strut \overline{\strut x_{i_1}x_{i_2}}x_{j_1}x_{j_2}}-\overline{\strut \overline{\strut x_{j_1}x_{j_2}}x_{i_1}x_{i_2}}+$$$$\overline{\strut (x_{i_1}\sh x_{j_1})x_{i_2}x_{j_2}}-\overline{\strut (x_{i_1}\sh x_{j_2})x_{i_2}x_{j_1}}+\overline{\strut (x_{i_2}\sh x_{j_2})x_{i_1}x_{j_1}}-\overline{\strut (x_{i_2}\sh x_{j_1})x_{i_1}x_{j_2}}.$$

Below we present a proof of the statement when skew-right-commutative elements have degrees at least 3. The case when one of skew-right-commutative element  has degree 2 can be proved by a similar way. We have
$$ [\overline{\strut x_{i_1}\cdots x_{i_m}},\overline{\strut x_{j_1}\cdots x_{j_n}}]=$$
$$\overline{\strut x_{i_1}\cdots x_{i_m}}\circ \overline{\strut x_{j_1}\cdots x_{j_n}}-\overline{\strut x_{j_1}\cdots x_{j_n}}\circ\overline{\strut x_{i_1}\cdots x_{i_m}}=$$
(by Lemma \ref{Skew to Skew})
$$\overline{\strut(x_{i_1}\cdots x_{i_m}\sh x_{j_1}\cdots x_{j_{n-2}})x_{j_{n-1}}x_{j_n}}-\overline{(x_{i_1}\cdots x_{i_{m-2}} x_{i_m}x_{i_{m-1}}\sh x_{j_1}\cdots x_{j_{n-2}})x_{j_{n-1}}x_{j_n}}+$$
$$(x_{i_1}\cdots x_{i_{m-1}}\sh x_{j_1}\cdots x_{j_{n-1}})x_{i_m}x_{j_n}-(x_{i_1}\cdots x_{i_{m-2}} x_{i_m}\sh x_{j_1}\cdots x_{j_{n-1}})x_{i_{m-1}}x_{j_n}-$$
$$(x_{i_1}\cdots x_{i_{m-1}}\sh x_{j_1}\cdots x_{j_{n-2}}x_{j_n})x_{i_m}x_{j_{n-1}}+(x_{i_1}\cdots x_{i_{m-2}}x_{i_m} \sh x_{j_1}\cdots x_{j_{n-2}}x_{j_n})x_{i_{m-1}}x_{j_{n-1}}-$$

$$\overline{\strut(x_{j_1}\cdots x_{j_n}\sh x_{i_1}\cdots x_{i_{m-2}})x_{i_{m-1}}x_{i_m}}+\overline{\strut(x_{j_1}\cdots x_{j_{n-2}} x_{j_n}x_{j_{n-1}}\sh x_{i_1}\cdots x_{i_{m-2}})x_{i_{m-1}}x_{i_m}}-$$
$$(x_{j_1}\cdots x_{j_{n-1}}\sh x_{i_1}\cdots x_{i_{m-1}})x_{j_n}x_{i_m}+(x_{j_1}\cdots x_{j_{n-2}} x_{j_n}\sh x_{i_1}\cdots x_{i_{m-1}})x_{j_{n-1}}x_{i_m}+$$
$$(x_{j_1}\cdots x_{j_{n-1}}\sh x_{i_1}\cdots x_{i_{m-2}}x_{i_m})x_{j_n}x_{i_{m-1}}-(x_{j_1}\cdots x_{j_{n-2}}x_{j_n} \sh x_{i_1}\cdots x_{i_{m-2}}x_{i_m})x_{j_{n-1}}x_{i_{m-1}}=$$
(by part {\bf a} of Proposition \ref{shuffle product})
$$\overline{\strut(x_{i_1}\cdots x_{i_m}\sh x_{j_1}\cdots x_{j_{n-2}})x_{j_{n-1}}x_{j_n}}-\overline{\strut(x_{i_1}\cdots x_{i_{m-2}} x_{i_m}x_{i_{m-1}}\sh x_{j_1}\cdots x_{j_{n-2}})x_{j_{n-1}}x_{j_n}}-$$
$$\overline{\strut(x_{j_1}\cdots x_{j_n}\sh x_{i_1}\cdots x_{i_{m-2}})x_{i_{m-1}}x_{i_m}}+\overline{(x_{j_1}\cdots x_{j_{n-2}} x_{j_n}x_{j_{n-1}}\sh x_{i_1}\cdots x_{i_{m-2}})x_{i_{m-1}}x_{i_m}}+$$
$$\overline{\strut(x_{i_1}\cdots x_{i_{m-1}}\sh x_{j_1}\cdots x_{j_{n-1}})x_{i_m}x_{j_n}}-\overline{\strut(x_{i_1}\cdots x_{i_{m-2}} x_{i_m}\sh x_{j_1}\cdots x_{j_{n-1}})x_{i_{m-1}}x_{j_n}}-$$
$$\overline{\strut(x_{i_1}\cdots x_{i_{m-1}}\sh x_{j_1}\cdots x_{j_{n-2}}x_{j_n})x_{i_m}x_{j_{n-1}}}+\overline{\strut(x_{i_1}\cdots x_{i_{m-2}}x_{i_m} \sh x_{j_1}\cdots x_{j_{n-2}}x_{j_n})x_{i_{m-1}}x_{j_{n-1}}}.$$ \end{proof}

\begin{lemma}\label{Lie is skew-rcom}
If $f\in ST(X),$ then $p(f)=-f.$
\end{lemma}

\begin{proof}
Since $ST(X)$ is generated by the commutator products on $X$ and $[x,y]=\overline{xy}$ for any $x,y\in X,$  Lemmas \ref{SkewRcomprudctgenerator} and \ref{commutator of Skews} complete the proof. \end{proof}

Now we prove that any skew-right-commutative element  of $Zin(X)$ is Lie. 

\begin{lemma}\label{SkewRcom is lie} Let $f\in Zin(X)$ with $p(f)=-f.$ Then $f\in ST(X).$
\end{lemma}

\begin{proof} Write $a \equiv b$ if $a-b\in ST(X).$ If $p(f)=-f,$ then $f$ can be written as a linear combination of skew-right-commutative elements. In order to prove that $f\in ST(X)$, it is sufficient to show that \begin{equation}\label{f4}\overline{\strut x_{i_1}\cdots x_{i_n}}\equiv 0.\end{equation} We prove it by induction on $n.$ 

If $n=2,$ then it is trivial, and if $n=3,$ then $$\overline{\strut x_{i_1}x_{i_2}x_{i_3}}=\frac{1}{2}[[x_{i_1},x_{i_2}],x_{i_3}]-\frac{1}{2}[[x_{i_1},x_{i_3}],x_{i_2}].$$ Assume that $(\ref{f4})$ is true for elements whose degree are less than $n.$ So $$\overline{\strut zx_{i_{k+1}}\cdots x_{i_n}}\equiv 0$$ for any 
Lie element $z$  whose degree is no more than $k.$
Set   $z:= \overline{\strut x_{i_1}\cdots x_{i_k}}$ and have
$$\overline{\strut\overline{\strut x_{i_1}\cdots x_{i_k}}x_{i_{k+1}}\cdots x_{i_n}}\equiv 0$$ for $1<k<n-1.$ Hence
$$\overline{\strut x_{i_1}\cdots x_{i_{k-1}}x_{i_k}x_{i_{k+1}}\cdots x_{i_n}}\equiv\overline{\strut x_{i_1}\cdots x_{i_k}x_{i_{k-1}}x_{i_{k+1}}\cdots x_{i_n}}.$$

Since the symmetric group $S_{n-2}$ is generated by transpositions $(12),(23),\ldots,(n-3\, n-2)$,
for any $\sigma\in {S_{n-2}}$ we have 
\begin{equation}\label{f5}\overline{\strut x_{i_1}\cdots x_{i_n}}\equiv\overline{x_{\sigma(i_1)}\cdots x_{\sigma(i_{n-2})}x_{i_{n-1}}x_{i_n}}.\end{equation}

By (\ref{f5}) and Lemma (\ref{SkewRcomprudctgenerator}) we have
$$[\overline{\strut x_{i_{1}}\cdots x_{i_{n-2}}x_{i_{n-1}}},x_{i_{n}}]\equiv\overline{\strut x_{i_{1}}\cdots x_{i_{n-3}}x_{i_{n-2}}x_{i_{n-1}} x_{i_{n}}}-\overline{\strut x_{i_{1}}\cdots x_{i_{n-3}}x_{i_{n-1}}x_{i_{n-2}} x_{i_{n}}}$$
$$-(n-2) \overline{\strut x_{i_{1}}\cdots x_{i_{n-3}}x_{i_{n}}x_{i_{n-2}}x_{i_{n-1}}}\equiv0
$$
and 
$$[\overline{\strut x_{i_{1}}\cdots x_{i_{n-2}}x_{i_{n}}},x_{i_{n-1}}]\equiv$$
$$\overline{\strut x_{i_{1}}\cdots x_{i_{n-3}}x_{i_{n-2}}x_{i_{n}} x_{i_{n-1}}}-\overline{\strut x_{i_{1}}\cdots x_{i_{n-3}}x_{i_{n}}x_{i_{n-2}} x_{i_{n-1}}}-(n-2) \overline{\strut x_{i_{1}}\cdots x_{i_{n-3}}x_{i_{n-1}}x_{i_{n-2}}x_{i_{n}}}=$$
$$-\overline{\strut x_{i_{1}}\cdots x_{i_{n-2}}x_{i_{n-1}} x_{i_{n}}}-\overline{\strut x_{i_{1}}\cdots  x_{i_{n-3}}x_{i_{n}}x_{i_{n-2}} x_{i_{n-1}}}-(n-2) \overline{\strut x_{i_{1}}\cdots x_{i_{n-3}} x_{i_{n-1}}x_{i_{n-2}}x_{i_{n}}}\equiv0.
$$
Consider sum of last two expressions and have  
$$[\overline{\strut x_{i_{1}}\cdots x_{i_{n-3}}x_{i_{n-2}}x_{i_{n-1}}},x_{i_{n}}]+
[\overline{\strut x_{i_{1}}\cdots x_{i_{n-3}}x_{i_{n-2}}x_{i_{n}}},x_{i_{n-1}}]=$$
$$-(n-1) \overline{\strut x_{i_{1}}\cdots x_{i_{n-3}}x_{i_{n-1}}x_{i_{n-2}}x_{i_{n}}}-(n-1) \overline{\strut x_{i_{1}}\cdots x_{i_{n-3}}x_{i_{n}}x_{i_{n-2}}x_{i_{n-1}}}\equiv0.$$
Thus 
$$ \overline{\strut x_{i_{1}}\cdots x_{i_{n-3}}x_{i_{n-1}}x_{i_{n-2}}x_{i_{n}}}\equiv- \overline{\strut x_{i_{1}}\cdots x_{i_{n-3}}x_{i_{n}}x_{i_{n-2}}x_{i_{n-1}}}.$$
In other words, 
\begin{equation}\label{f6}\overline{\strut x_{i_{1}}\cdots x_{i_{n-3}}x_{i_{n-2}}x_{i_{n-1}}x_{i_{n}}}\equiv- \overline{\strut x_{i_{1}}\cdots x_{i_{n-3}}x_{i_{n}}x_{i_{n-1}}x_{i_{n-2}}}.\end{equation}

Set $u=x_{i_1}\cdots x_{i_{n-4}}.$ By  (\ref{f5}) and (\ref{f6}) we have  
 $$\overline{\strut ux_{i_{n-3}}x_{i_{n-2}}x_{i_{n-1}}x_{i_{n}}}\equiv- \overline{\strut ux_{i_{n-3}}x_{i_{n}}x_{i_{n-1}}x_{i_{n-2}}}\equiv  \overline{\strut ux_{i_{n}}x_{i_{n-3}}x_{i_{n-2}}x_{i_{n-1}}}\equiv $$$$\overline{\strut u x_{i_{n-1}}x_{i_{n}}x_{i_{n-3}}x_{i_{n-2}}}\equiv  \overline{\strut ux_{i_{n}}x_{i_{n-1}}x_{i_{n-3}}x_{i_{n-2}}}\equiv \overline{\strut u x_{i_{n-2}}x_{i_{n}}x_{i_{n-1}}x_{i_{n-3}}}\equiv  $$
$$\overline{\strut ux_{i_{n-3}}x_{i_{n-2}}x_{i_{n}}x_{i_{n-1}}}\equiv - \overline{\strut u x_{i_{n-3}}x_{i_{n-2}}x_{i_{n-1}}x_{i_{n}}}.$$
Hence $$\overline{\strut ux_{i_{n-3}}x_{i_{n-2}}x_{i_{n-1}}x_{i_{n}}}\equiv0$$ and this completes the proof.
\end{proof} 

\subsection{ Proof of Theorem \ref{Lie Criterion}. }

It follows from Lemmas \ref{Lie is skew-rcom} and \ref{SkewRcom is lie}.$\square$

\subsection{ Proof of Theorem \ref{Skew element basis}. }
Since a skew-rcom element defined as $$\overline{\strut x_{i_1}\cdots x_{i_n}}= x_{i_1}\cdots x_{i_n}-x_{i_1}\cdots x_{i_{n-2}}x_{i_n}x_{i_{n-1}}$$  a difference of two base elements any linear combination of skew-rcom elements is trivial, hence they are linear independent in $Zin(X).$ 
By Lemma \ref{Lie is skew-rcom} any element of $ST(X)$ is  a linear combination of skew-rcom elements. 
So we have proved that the set of skew-rcom elements, generated by set $X,$ forms a base of $ST(X).$ 

Let us count the number of skew-rcom elements of degree $n$ generated by $x_1,\ldots,x_q$ in which $x_1,\ldots,x_q$ occur $m_1,\ldots,m_q$ times, respectively.  Consider skew-rcom elements whose last two elements are $x_i, x_j$ for $i<j.$  Then the number of such type of skew-rcom elements of degree $n$ equals
$$\frac{(n-2)!}{m_1! \cdots (m_i-1)!\cdots (m_j-1)!\cdots m_q!}=\frac{(n-2)!}{m_1!\cdots m_{q}!}m_im_j,$$ where $m_1+\cdots+m_q=n.$ 
Hence $$dim\,ST(X)_{m_1,\ldots,m_q}=\sum_{i<j}\frac{(n-2)!}{m_1!\cdots m_q!}m_im_j.$$ If $m_i=1$ for all $i,$ then $n=q,$ and, $$dim\,ST(X)_{1,\ldots,1}=\sum_{1\le i< j\le q} 
(q-2)!=q!/2.$$  $\square$

\subsection{ Proof of Corollary \ref{On last two elements in skew rcom}.} We present a proof of our Corollary for $abc-acb.$ The case $bc-cb$ can be established in a similar way. 

Let $a\in Zin(X)$ and suppose $b,c\in ST(X)$. Then by Theorem \ref{Skew element basis} 
$$b=\sum_\alpha\lambda_\alpha\overline{x_\alpha}, \quad c=\sum_\beta\mu_\beta\overline{y_\beta}.$$ 
We have 
$$abc-acb=a\sum_\alpha\lambda_\alpha\overline{x_\alpha}\sum_\beta\mu_\beta\overline{y_\beta}-a\sum_\beta\mu_\beta\overline{y_\beta}\sum_\alpha\lambda_\alpha\overline{x_\alpha}=$$
$$\sum_{\alpha,\beta}\lambda_{\alpha}\mu_\beta(a\overline{x_\alpha}\,\overline{y_\beta}-a\overline{y_\beta}\,\overline{x_\alpha}).$$ Let $x_{\alpha}=x_{i_1}\cdots x_{i_m}$ and $y_\beta=y_{j_1}\cdots y_{j_n}.$ The cases, when at least one of $m$ and $n$ is equal to two, can be easily proved. Suppose $m,n\geq3.$ Set $u=x_{i_1}\cdots x_{i_{m-2}}$ and $v=y_{j_1}\cdots y_{j_{n-2}}.$ $$a\overline{\strut x_{\alpha}}\,\overline{\strut y_{\beta}}=$$$$a(ux_{i_{m-1}}x_{i_m})(vy_{j_{n-1}}y_{i_n})-a(ux_{i_m}x_{i_{m-1}})(vy_{j_{n-1}}y_{i_n})-$$$$a(ux_{i_{m-1}}x_{i_m})(vy_{i_n}y_{j_{n-1}})+a(ux_{i_m}x_{i_{m-1}})(vy_{i_n}y_{j_{n-1}})=$$
(by part {\bf b} of Proposition \ref{shuffle product})
$$((a\sh ux_{i_{m-1}})x_{i_m}\sh vy_{j_{n-1}})y_{j_n}-((a\sh ux_{i_m})x_{i_{m-1}}\sh vy_{j_{n-1}})y_{j_n}-$$$$((a\sh ux_{i_{m-1}})x_{i_m}\sh vy_{j_n})y_{j_{n-1}}+((a\sh ux_{i_m})x_{i_{m-1}}\sh vy_{j_n})y_{j_{n-1}}$$
(by part {\bf c} of Proposition \ref{shuffle product})
$$((a\sh ux_{i_{m-1}})\sh vy_{j_{n-1}})x_{i_m}y_{j_n}+((a\sh ux_{i_{m-1}})x_{i_m}\sh v)y_{j_{n-1}}y_{j_n}-$$
$$((a\sh ux_{i_m})\sh vy_{j_{n-1}})x_{i_{m-1}}y_{i_n}-((a\sh ux_{i_m})x_{i_{m-1}}\sh v)y_{j_{n-1}}y_{j_n}-$$
$$((a\sh ux_{i_{m-1}})\sh vy_{i_n})x_{i_m}y_{j_{n-1}}-((a\sh ux_{i_{m-1}})x_{i_m}\sh v)y_{i_n}y_{j_{n-1}}+$$
$$((a\sh ux_{i_m})\sh vy_{i_n})x_{i_{m-1}}y_{j_{n-1}}+((a\sh ux_{i_m})x_{i_{m-1}}\sh v)y_{i_n}y_{j_{n-1}}=$$
$$\overline{\strut((a\sh ux_{i_{m-1}})x_{i_m}\sh v)y_{j_{n-1}}y_{j_n}}-\overline{\strut((a\sh ux_{i_m})x_{i_{m-1}}\sh v)y_{j_{n-1}}y_{j_n}}+$$
$$((a\sh ux_{i_{m-1}})\sh vy_{j_{n-1}})x_{i_m}y_{j_n}-((a\sh ux_{i_m})\sh vy_{j_{n-1}})x_{i_{m-1}}y_{i_n}-$$
$$((a\sh ux_{i_{m-1}})\sh vy_{j_n})x_{i_m}y_{j_{n-1}}+((a\sh ux_{i_m})\sh vy_{j_n})x_{i_{m-1}}y_{j_{n-1}}.$$
By similar way one can have $$a\overline{\strut y_{\beta}}\,\overline{\strut x_{\alpha}}=$$
$$\overline{\strut((a\sh vy_{j_{n-1}})y_{j_n}\sh u)x_{i_{m-1}}x_{i_m}}-\overline{\strut ((a\sh vy_{j_n})y_{j_{n-1}}\sh u)x_{i_{m-1}}x_{i_m}}+$$$$((a\sh vy_{j_{n-1}})\sh ux_{i_{m-1}})y_{j_n}x_{i_m}-((a\sh vy_{j_n})\sh ux_{i_{m-1}})y_{j_{n-1}}x_{i_m}-$$$$((a\sh vy_{j_{n-1}})\sh ux_{i_m})y_{j_n}x_{i_{m-1}}+((a\sh vy_{j_n})\sh ux_{i_m})y_{j_{n-1}}x_{i_{m-1}}.$$
So $$a\overline{\strut x_{\alpha}}\,\overline{\strut y_{\beta}}-a\overline{\strut y_{\beta}}\,\overline{\strut {x_\alpha}}=$$(by part {\bf a} of Proposition \ref{shuffle product})
$$\overline{\strut((a\sh ux_{i_{m-1}})x_{i_m}\sh v)y_{j_{n-1}}y_{j_n}}-\overline{\strut((a\sh ux_{i_m})x_{i_{m-1}}\sh v)y_{j_{n-1}}y_{j_n}}-$$
$$\overline{\strut((a\sh vy_{j_{n-1}})y_{j_n}\sh u)x_{i_{m-1}}x_{i_m}}+\overline{\strut ((a\sh vy_{j_n})y_{j_{n-1}}\sh u)x_{i_{m-1}}x_{i_m}}+$$
$$\overline{\strut((a\sh ux_{i_{m-1}})\sh vy_{j_{n-1}})x_{i_m}y_{j_n}}-\overline{\strut((a\sh ux_{i_m})\sh vy_{j_{n-1}})x_{i_{m-1}}y_{i_n}}-$$
$$\overline{\strut((a\sh ux_{i_{m-1}})\sh vy_{j_n})x_{i_m}y_{j_{n-1}}}+\overline{\strut((a\sh ux_{i_m})\sh vy_{j_n})x_{i_{m-1}}y_{j_{n-1}}}.$$
Hence by Theorem \ref{Lie Criterion} $abc-acb\in ST(X).$ $\square$

\section{\label{nn}\ Proof of Theorem \ref{a Jordan criterion}}
In this section we prove Jordan criterion for $Zin(X).$ 
\begin{lemma}\label{a Jordan element in a free Zinbiel algebra}
$\{\{\cdots\{x_{i_1},x_{i_2}\}\cdots\},x_{i_n}\}=\sum_{\sigma\in S_n}x_{\sigma(i_1)}x_{\sigma(i_2)}\cdots x_{\sigma(i_{n})}.$
\end{lemma}
\begin{proof}
We prove it by induction on $n.$ If $n=2,$ then $\{x_{i_1},x_{i_2}\}=x_{i_1}x_{i_2}+x_{i_2}x_{i_1}.$ Suppose that it is true for $n-1.$ Then 
$$\{\{\cdots\{x_{i_1},x_{i_2}\}\cdots\},x_{i_n}\}=$$
$$\{\{\cdots\{x_{i_1},x_{i_2}\}\cdots\},x_{i_{n-1}}\}x_{i_n}+x_{i_n}\{\{\cdots\{x_{i_1},x_{i_2}\}\cdots\},x_{i_{n-1}}\}=$$
(by induction hypothesis)
$$\sum_{\sigma\in S_{n-1}}x_{\sigma(i_1)}x_{\sigma(i_2)}\cdots x_{\sigma(i_{n-1})}x_{i_n}+x_{i_n}\sum_{\sigma\in S_{n-1}}x_{\sigma(i_1)}x_{\sigma(i_2)}\cdots x_{\sigma(i_{n-1})}=$$
(by Proposition \ref{Loday's prop})
$$\sum_{\sigma\in S_n}x_{\sigma(i_1)}x_{\sigma(i_2)}\cdots x_{\sigma(i_{n})}.$$
\end{proof}
\subsection{ Proof of Theorem \ref{a Jordan criterion}. }
Recall that $A^{(+)}$ is associative and commutative algebra if $A$ is Zinbiel. Any Jordan element in $Zin(X)$ can be written as linear combination of left-normed Jordan monomials in $X$ by anti-commutators. Then the proof follows from  Lemma \ref{a Jordan element in a free Zinbiel algebra} and definition of the map $D.$

Let $\varphi:K[X]\rightarrow J(X)$ be a canonical homomorphism from polynomial algebra generated by $X$ to $J(X)$ defined as $x_{i_1}x_{i_2}\cdots x_{i_n}\mapsto\{\{\cdots\{x_{i_1},x_{i_2}\}\cdots\}x_{i_n}\}.$ Then it is clear that $Ker\,\varphi$ is zero and therefore $K[X]$ and $J(X)$ are isomorphic.$\square$

Denote by  $J(X)_{m_1,\ldots,m_q}$ the homogenous part of $J(X)$ generated by $m_i$ generators $x_i$ where $i=1,\ldots,q.$ 
\begin{corollary}\label{dimension of space of Jordan elements}
The dimension of the homogenous part $J(X)_{m_1,\ldots,m_q}$ of $J(X)$ is equal to $dim\,J(X)_{m_1,\ldots,m_q}=1.$ 
\end{corollary}
\begin{proof}
It is an immediate consequence of Theorem \ref{a Jordan criterion}.
\end{proof}

\section{\label{nn}\ Speciality of  $T(\{x,y\})$}

In this section we prove that the free Tortkara algebra on two generators $T(\{x,y\})$ is special. As a corollary, we obtain construction of a base of $T(\{x,y\})$ in terms of left-normed elements.
\begin{lemma}\label{left-normed elements on 2 generators}
 Let $T_n$ be the n-th homogenous part of $T(\{x,y\}).$  Then $T_{n+1}=T_nT_1$ for any $n.$
\end{lemma} 
\begin{proof}
Clearly, $T_{n+1}\supseteq T_nT_1.$ 

We write $a\equiv b$ if $a-b\in T_nT_1.$  We prove the statement by induction on degree $n.$ We have
$$(ab)(cd)=\frac{1}{2}J(b,c,d)a-\frac{1}{2}J(a,c,d)b-\frac{1}{2}J(a,b,d)c+\frac{1}{2}J(a,b,c)d\equiv0$$
This is the basis of induction for $n.$ Suppose that our statement is true for fewer than $n>4.$ Let $C\in T_n$ and $C=A_{k}B_{l}$ where $A_{k}$ and $B_{l}$ are elements of $T(\{x,y\})$ whose degrees are $k$ and $l,$ respectively, and $k+l=n.$
Now we consider induction on $l.$ By induction on $n$ we may assume that they are left-normed and write
$$C=A_{k}B_{l}=(A_{k-1}a_{k})(B_{l-1}b_{l})$$ where $a_{k},b_{l}\in\{x,y\}.$
Suppose $l=2$ and $b_1=x,$ $b_2=y.$ Assume $a_k=x.$ Then by identity (\ref{Tortkara identity}) and induction on $n$ we have
$$C=(A_{k-1}x)(yx)=J(A_{k-1},x,y)x\equiv 0.$$
Suppose that our statement is true for fewer than $l>2.$
We have 
$$(A_{k-1}a_{k})(B_{l-1}b_{l})=$$
(by anticommutativity identity)
$$-(A_{k-1}a_{k})(b_{l}B_{l-1})=$$
(by identity (\ref{linear-Tortkara identity}))
$$(A_{k-1}B_{l-1})(b_{l}a_k)-J(A_{k-1},a_k,b_l)B_{l-1}-J(A_{k-1},B_{l-1},b_l)a_k.$$
We note that by base of induction on $l$ $(A_{k-1}B_{l-1})(b_{l}a_k)\equiv 0,$ and $J(A_{k-1},a_k,b_l)B_{l-1}\equiv 0.$ By induction on $n$ we have $J(A_{k-1},B_{l-1},b_l)a_k\equiv 0.$
Hence $$(A_{k-1}a_{k})(B_{l-1}b_{l})\equiv 0.$$
\end{proof}
\subsection{ Proof of Theorem \ref{Speciality of $T({x,y})$}.} It is sufficient to show that algebras $T(\{x,y\})$ and $ST(\{x,y\})$ are isomorphic. Let $\varphi$ be a natural homomorphism from $T(\{x,y\})$ to $ST(\{x,y\}).$ By Lemma \ref{left-normed elements on 2 generators} the vector space $T(\{x,y\})$ is spanned by the set of  left-normed elements. We note that number of left-normed elements in two generators is equal to the number of skew-rcom elements in two generators. Suppose that the kernel of $\varphi$ is not zero. Then we have a linear combination of skew-rcom elements which is zero in $ST(\{x,y\}).$ It contradicts to the first part of Theorem \ref{Skew element basis}. Therefore, $Ker\,\varphi=(0).$ $\square$

\begin{corollary}
Set of left-normed elements forms a base of $T({x,y}).$
\end{corollary}

\section{\label{nn}\ Speciality of homomorphic images of $ST(\{x,y\})$}
Let $\alpha$ be an ideal of $ST(X).$ By Cohn's criterion (Theorem 2.2 of \cite{Cohn}) $ST(X)/\alpha$ is special if and only if $\{\alpha\}\cap ST(X)\subseteq \alpha$ where $\{\alpha\}$ is the ideal of $Zin(X)$ generated by the set $\alpha.$ 

\subsection{ Proof of Theorem \ref{speciality criterion on two generators}.} Assume that $g_i$ $(i\in I)$ are generators of the ideal $\alpha.$  It is clear that if $\overline{xy}\in \alpha$ then $ST(\{x,y\})/\alpha$ is special.

Therefore, by Theorem \ref{Lie Criterion} we can  assume that each element $g_i$ has a form $\overline{f_ixy}$ for some $f_i\in Zin(\{x,y\}).$

Let $w$ be a non-zero element of $\{\alpha\}\cap ST(\{x,y\})$. Then $p(w)=-w$ and $w$ is a linear combination of left-normed monomials in $x, y, g_i (i\in I)$  such that each monomial is linear by at least one generator of $\alpha$. Let $a_1\cdots a_n$ be a term of $w$ in the linear combination. To prove the statement we consider two cases, depending on what position a generator appear in  $a_1\cdots a_n.$ 

{\it Case 1.} Suppose that generators of $\alpha$ appear only in the first $n-2$ positions in $a_1\cdots a_n.$  Then write all $a_i's$ in terms of elements of $X.$ Since $w\in ST(\{x,y\})$, $w$ must have the term  $p(a_1\cdots a_n)$ with opposite sign. Hence $\overline{a_1\cdots a_n}\in\alpha.$

{\it Case 2.} Suppose that generators of $\alpha$ appear in either $n-1$-th or $n$-th positions in $a_1\cdots a_n,$ (a generator of $\alpha$ may appear in the first $n-2$-positions), namely,
\begin{center}
$a_{n-1}=x$ and $a_n=\overline{\strut f_ixy},$ or $a_{n-1}=\overline{\strut f_ixy}$ and $a_n=x$ 
\end{center}
for some $i.$ If generators of $\alpha$ appear in both $n-1$-th and $n$-th positions of $a_1\cdots a_n,$ then write one of them in terms of $x's$ and $y's.$ We also express $a_1,\ldots,a_{n-2}$ in terms of $x's$ and $y's$, therefore we can assume that $a_1,\ldots, a_{n-2}\in X.$ Let us denote $a_1\cdots a_{n-2}$ by $u$.

Now we show that if $ux\overline{\strut  f_ixy}$ is a term of $w$ then $w$ has the term $u\overline{\strut  f_ixy}x$ with opposite sign. 

We have $$ux\overline{\strut  f_ixy}=R_1+R_2-R_3,$$
where $$R_1=\overline{\strut  f_iuxxy}+\overline{\strut  uf_ixxy}+\overline{\strut  uxf_ixy}, \quad R_2=f_iuxxy+f_ixuxy+uf_ixxy,$$$$ R_3=f_iuyxx+f_iyuxx+uf_iyxx.$$
By Theorem \ref{Lie Criterion}, $R_1$ is Lie, but $R_2,R_3$ are not Lie.
Since $w\in ST(\{x,y\}),$ the term $R_3$ should be cancelled and for each term $s\in\{f_iuxxy,f_ixuxy,uf_ixxy\}$ of $R_2,$ $w$ must have  terms $s$ or $p(s)$ with opposite sign to cancel $s$ or have $\overline{s}.$ Therefore, $w$ must have some terms in which $g_i$ $(i\in I)$ appear in either $n-1$-th or $n$-th positions. These kind of terms are generated by $u,f,x,x,y.$ Then all possibilities of such types are $u\overline{\strut  f_ixy}x,f_ix\overline{\strut  uxy}$ and $f_i\overline{\strut  uxy}x.$ We have
$$u\overline{\strut  f_ixy}x=f_iuxyx+f_ixuyx+uf_ixyx-f_iuyxx-f_iyuxx-uf_iyxx,$$
$$f_ix\overline{\strut  uxy}=\overline{\strut  f_iuxxy}+\overline{\strut  f_ixuxy}+\overline{\strut  uf_ixxy}+f_iuxxy+uf_ixxy+uxf_ixy-$$
$$f_iuyxx-uf_iyxx-uyf_ixx,$$
$$f_i\overline{\strut  uxy}x=f_iuxyx+uf_ixyx+uxf_iyx-f_iuyxx-uf_iyxx-uyf_ixx.$$
We see that the element $f_iyuxx$ is a term of only $u\overline{\strut  f_ixy}x$ and moreover,
$$ux\overline{\strut f_ixy}-u\overline{\strut f_ixy}x=\overline{\strut f_iuxxy}+\overline{\strut uf_ixxy}+\overline{\strut uxf_ixy}+$$$$\overline{\strut f_iuxxy}+\overline{\strut f_ixuxy}+\overline{\strut uf_ixxy}\in ST(\{x,y\}).$$
Therefore, $w$ has the term $u\overline{\strut f_ixy}x$ with opposite sign. Since $\overline{\strut f_ixy}$ is a generator of $\alpha,$ and by Corollary \ref{On last two elements in skew rcom}
$$ux\overline{\strut f_ixy}-u\overline{\strut f_ixy}x\in\alpha.$$
Hence $a_1\cdots a_n-a_1\cdots a_{n-2}a_na_{n-1}\in\alpha.$
If $x\overline{f_ixy}$ is  a nonzero term of $w$, then by similar way one can show that $w$ must have term $\overline{\strut f_ixy}x$. 

So we obtain $w\in\alpha.$ It follows $\{\alpha\}\cap ST(\{x,y\})\subseteq \alpha.$ Hence by Cohn's criterion $ST(\{x,y\})/\alpha$ is special.  

Now we show that a homomorphic image of $ST(\{x,y,z\})$ may be not special.

Let $\alpha$ be an ideal of $ST(\{x,y,z\})$ generated by elements $$g_{1}=\overline{\strut yyz},g_2=\overline{\strut yxz},g_3=\overline{\strut yxy}.$$
Consider an element $$w=\overline{\strut xyyz}-\overline{\strut yyxz}+\overline{\strut zyxy}.$$
Then $$w=xg_1-yg_2+zg_3.$$
It follows $$w\in\{\alpha\}\cap ST(\{x,y,z\}).$$
One can easily check that there are no $\lambda_1,\lambda_2,\lambda_3\in {\bf K}$ so that
$$w=\lambda_1[x,g_1]+\lambda_2[y,g_2]+\lambda_3[z,g_3].$$ 
Then $w\notin\alpha.$ Hence by Cohn's criterion, $ST(\{x,y,z\})/\alpha$  is not special.$\square$

\section{\label{nn}\ Remarks and open questions}
{\bf 1.} Let $A={\bf C}[x]$ be an algebra with multiplication \begin{equation}\label{f0}a\star b=b\int_0^x(\int^x_0a\,dx)dx.\end{equation} $(A,\star)$ is not a Zinbiel algebra.
This algebra was considered in \cite{Dzhumadil'daev}. It was proved that it satisfies the following identities $$a\star(b\star c)-b \star (a\star c)=0,$$ $$([a,b],c,d)+([b,c],a,d)+([c,a],b,d)=0$$
where $(a,b,c)=a\star (b\star c)-(a\star b)\star c.$
Moreover, it was proved that algebra $A$ with respect to commutator $[a,b]_{\star}=a\star b-b\star a$ is a Tortkara algebra. A question on speciality of $(A ,[,]_{\star}) $ was posed.

We show that answer is positive.
Let $B={\bf C}[x]$ be an algebra with multiplication $$a\diamond b=b\int_0^x(\int^x_0a\,dx)dx+\int^x_0 a\,dx\int^x_0b\,dx.$$ Then $(B,\diamond)$ is a Zinbiel algebra. For $\diamond$ multiplication we define commutator $[a,b]_{\diamond}=a\diamond b-b\diamond a.$ 
Note that $[a,b]_{\star}=[a,b]_{\diamond}.$ So $A^{(-)}$ is isomorphic to $B^{(-)}.$ Hence $(A,[,]_{\star})$ is special.  

{\bf 2.} It is shown in \cite{Dzhumadil'daev} that an algebra with identities (\ref{Tortkara example}) is not Zinbiel but under the commutator product is Tortkara.
What about speciality of these algebras?

{\bf 3.} Let $k_m$ be kernel of the natural homomorphism from free Tortkara algebra  to free special Tortkara algebra on $m$ generators. An element of the ideal $k_m$ is called a $s$-identity. We showed that $k_2=(0)$. Are there $s$-identities for  $m>2$?

{\bf 4.} Is it true the analogue of Lemma \ref{left-normed elements on 2 generators} for $m>2$ generators? Whenever it is valid for $m$ generators, it immediately follows speciality of $T(\{x_1,\ldots,x_m\}),$ in particular, $k_m=(0)$.

{\bf 5.} What is the analogue of classical Artin's theorem for Tortkara algebras? In other words, if $A$ is a Tortkara algebra and $alg<a,b>$ is the subalgebra generated by $a,b\in A,$ then what is the characterization of $Var(alg<a,b>)?$

{\bf Acknowledgments.} The authors are grateful to Professor I.P. Shestakov for his discussion and for essential comments.

\end{document}